\documentclass[12pt,amstex]{amsart}

\usepackage{mathptmx}
\usepackage{mathrsfs}
\usepackage{verbatim}
\usepackage{url}
\usepackage[all]{xy}
\usepackage{color}

\usepackage{tikz}
\usetikzlibrary{arrows.meta,animations}
\tikzset{>=stealth}
\usetikzlibrary{calc} 
\usetikzlibrary{scopes}

\usepackage[colorlinks=true,citecolor=blue]{hyperref}
\usepackage{stmaryrd}
\usepackage{epsfig}
\usepackage{amsmath,bm}
\usepackage{amssymb}
\usepackage{amssymb,amsthm,amsmath}
\usepackage{mathrsfs}
\usepackage{amscd}
\usepackage{graphicx}
\usepackage{pstricks}
\usepackage{theoremref}

\theoremstyle{plain}
\newtheorem{claim}{Claim}
\newtheorem{thm}{Theorem}[section]
\newtheorem{lem}[thm]{Lemma}
\newtheorem{prop}[thm]{Proposition}

\theoremstyle{definition}
\newtheorem{defn}[thm]{Definition}
\newtheorem{rem}[thm]{Remark}

\newcommand{\Z}{{\mathbb{Z}}}
\newcommand{\N}{\mathbb{N}}
\newcommand{\F}{\mathcal F}

\newcommand{\ep}{\varepsilon}
\newcommand{\ra}{\rightarrow}

\def \RP {{\bf RP}}

\newcommand{\E}{{\mathbb{E}}}

\def \A {\mathcal A}
\def \d {\delta}

\begin{document}
	\title{SATURATION FOR PRODUCT SYSTEMS of polynomials}
	
	\author{Qinqi Wu, Jiaqi Yu}
	
	\address{School of Mathematics, Shanghai university of Finance and Economics, Shanghai, 200433, P.R. China}
	\email{wuqinqi@mail.shufe.edu.cn}
	
	\address{}
	
	\email{}


	\thanks{ }
	
	\date{\today}
	
	\begin{abstract}
Recently, Qiu, Xu, Ye and Yu proved that for product system of finitely many minimal systems, the maximal $\infty$-step pro-nilfactor of the system is the topological characteristic factor. 
In this paper, we extend the result to polynomials.
	\end{abstract}
	
	\maketitle
	
	\section{Introduction}
	      
	In this section, we will provide the background of the research and state our main results. By a topological dynamical system (t.d.s. for short) $(X,T)$, we mean that $X$ is a compact metric
	space and $T:X\to X$ is a homeomorphism.
	
	\subsection{Motivation}
	The study of the convergence of {\it multiple ergodic averages}:
	\begin{equation}\label{ave1}
		\frac{1}{N} \sum_{n = 0}^{N - 1} f_1(T^n x) f_2(T^{2n} x) \dots f_d(T^{dn} x),
	\end{equation}
	began with Furstenberg's proof of
	Szemer\'{e}di's Theorem \cite{Sz75} via an ergodic theoretical analysis \cite{FH}. In the study the notion of characteristic factor, to control the behavior of the limit in the sense of $L^2$ form, was introduced by Furstenberg and Weiss. 
	By the results of Host-Kra \cite{HK05} and Ziegler \cite{TZ07}, the characteristic factor of (\ref{ave1}) is some $(d-1)$-step pro-nilsystem,
	which is also called a \emph{pro-nilfactor}.
	
	In the topological setting, the corresponding notion of  characteristic factor was first introduced by Glasner in \cite{GE94}. To be precise, let $\pi:(X,T)\ra(Y,T)$ be a factor map of t.d.s. and $d\in\N$, $(Y,T)$ is a {\em $d$-step topological characteristic factor (TCF for short)} if there exists a dense $G_{\delta}$ set $\Omega$ of $X$ such that for each $x\in\Omega$ the orbit closure $L_x^d=\overline{\mathcal{O}}(x^{(d)},\tau_d)$ is $\pi^{(d)}=:\pi\times\cdots \times \pi$ ($d$-times) saturated where $x^{(d)}=(x,\ldots,x)$ ($d$-times) and $\tau_d=T\times T^2\times \cdots \times T^d$, i.e., $(\pi^{(d)})^{-1}\pi^{(d)}(L_x^d)=L_x^d$.
	
	In\cite{HKM}, Host, Kra and Maass introduced the notion of {\it $d$-step regionally proximal relation}, denoted by $\mathbf{RP}^{[d]}$, to get the corresponding pro-nilfactors for t.d.s.. For $d \in \mathbb{N}$, we say that a minimal system $(X,T)$ is a $d$-step pro-nilsystem if $\mathbf{RP}^{[d]} = \Delta$. 
	For a minimal distal system $(X,T)$, it was proved that $\mathbf{RP}^{[d]}$ is an equivalence relation and $X_d:=X/\mathbf{RP}^{[d]}$ is the maximal $d$-step pro-nilfactor \cite{HKM}. Later, Shao and Ye \cite{SY} showed that in fact for any minimal system the regionally proximal relation of order $d$ is an equivalence relation. Furthermore, for any minimal system $\mathbf{RP}^{[\infty]} = \bigcap_{d \geqslant 1} \mathbf{RP}^{[d]}$ is also an equivalence relation.

	Let $\pi:(X,T)\to (Y,T)$ be a factor map of  t.d.s.  and $d\in\mathbb{N}$.
	We say that $(Y,T)$ is a {\it $d$-step weak topological characteristic factor (wTCF for short)} of $(X,T)$ if there exists a dense $G_\delta$ subset $\Omega$ of $X$ such that $(\pi^{-1}(\pi(x))^d \subseteq L_x^d(X)$ for any $x\in \Omega$, and say $(Y,T)$ is an $\infty$-step wTCF of $(X,T)$ if it is a $d$-step wTCF for every $d\in\N$. Glasner, Hunag, Shao, Weiss and Ye in \cite{GHSWY} obtained the following theorem.
	
	\medskip

	\noindent{\bf Theorem GHSWY:} {\it For any minimal system $(X,T)$ and $d\in\mathbb{N}$, $(X_\infty,T)$ is a $d$-step wTCF of $(X,T)$ where $X_\infty=X/\RP^{[\infty]}$ is the maximal $\infty$-step pro-nilfactor of $X$.}

	\medskip
	Recently, Theorem GHSWY was extended to the product system of finitely	many minimal systems by Qiu, Xu, Ye and Yu \cite{QXYY25}.

	\medskip
	\noindent{\bf Theorem QXYY:}
	{\it
		Let $k\in\N$ and $(X_i,T_i)$ be minimal systems for any $1\leq i\leq k$. For $d\in \mathbb{N}\cup\{\infty\}$, 
		$(\prod_{i=1}^{k}X_{i,d},\prod_{i=1}^{k}T_{i})$ is a $d$-step wTCF of $(\prod_{i=1}^{k}X_{i},\prod_{i=1}^{k}T_{i})$ where $X_{i,d}$ be the maximal $d$-step pro-nilfactor of $X_i$.}

	\medskip
	Now let $C=\{p_1,\dots,p_d\}$ be a set of integral polynomials vanishing at 0. It is easy to extend the notion of  $d$-step TCF (resp. wTCF) to {\it $d$-step TCF (resp. wTCF) for $C$} by replacing $L_x^d$ with $$L^C_x:=\overline{\{(T^{p_1(n)}x, \ldots, T^{p_d(n)}x): n\in\Z\}}.$$ 
    This definition was inspired by the study of the multiple ergodic averages
    \begin{equation}
		\frac{1}{N} \sum_{n = 0}^{N - 1} f_1(T^{p_1(n)} x) f_2(T^{p_2(n)} x) \dots f_d(T^{p_d(n)} x).
	\end{equation}
	In \cite{Qiu}, Qiu proved a version of the saturation theorem for polynomials and this result was subsequently  strengthened by Huang, Shao and Ye \cite{HSY25}. By utilizing the definition of wTCF, we have

	\medskip
	\noindent{\bf Theorem HSY:}
	{\it Let $C=\{p_1,\dots,p_d\}$ be a set of integral polynomials vanishing at $0$. For any minimal system $(X,T)$ and $d\in\mathbb{N}$, $(X_\infty,T)$ is a $d$-step wTCF for $C$.}

	\subsection{Main results}
	In this paper, we demostrate the saturation theorem for polynomials to the product system. 
	The following is one of our main results. 
	
	\begin{thm}\label{main thm}
			Let $\pi_{i}: (X_i,T_i)\rightarrow (Y_i,T_i)$ be factor maps of minimal systems for every $1\leq i\leq k$ and let $p_1, p_2, \ldots, p_d$ be distinct non-constant integral polynomials vanishing at $0$. If $R_{\pi_{i}}\subseteq \mathbf{RP}_{X_i}^{[\infty]}$ for every $1\leq i\leq k$,
			then for any $d\in \mathbb{N}$ and any open subsets $V_{0}^{(i)},V_{1}^{(i)},\ldots, V_d^{(i)}$ of $X_i$ 
			with $\bigcap_{j=0}^d\pi_{i}(V_{j}^{(i)})\neq\emptyset$,
			there exists $n\in\mathbb{N}$ such that
			\[
			V_{0}^{(i)}\cap T_i^{-p_1(n)} V_{1}^{(i)}\cap \cdots \cap T_i^{-p_d(n)} V_d^{(i)}\neq\emptyset,\ \ \ \forall \; 1\leq i\leq k.
			\]
		\end{thm}

	As a corollary of Theorem \ref{main thm}, we have the following theorem.
	\begin{thm}\label{main thm'}
		Let $\pi_{i}: (X_i,T_i)\rightarrow (X_{i,\infty},T_i)$ be factor maps of minimal systems for every $1\leq i\leq k$ and let $C=\{p_1, p_2, \ldots, p_d\}$ be distinct non-constant integral polynomials vanishing at $0$. Then $(\prod_{i=1}^{k}X_{i,\infty},\prod_{i=1}^{k}T_i)$ is a $d$-step wTCF for $C$ of $(\prod_{i=1}^{k}X_{i},\prod_{i=1}^{k}\prod_{i=1}^{k}T_i)$ for any $d\in\mathbb{N}$  where $X_{i,\infty}$ be the maximal $\infty$-step pro-nilfactor of $X_i$.
	\end{thm}
	
	Indeed, We prove a more general version of this result (Theorem \ref{generalization of main1}): Let $(X_{i,l_i},T_i)$ be the $d$-step wTCF for $C$ of $(X_i,T_i)$ for any $1\leq i\leq k$ where $X_{i,l}$ be the maximal $l$-step pro-nilfactor of $X_i$ respectively, then $(\prod_{i=1}^{k}X_{i,l_i},\prod_{i=1}^{k}T_i)$ is a $d$-step wTCF for $C$ of $(\prod_{i=1}^{k}X_{i},\prod_{i=1}^{k}T_i)$.

    \medskip
	In the local entropy theory, Theorem \ref{main thm} gives an application in the independent tuple along the family of polynomials. The notion of independence was firstly introduced and studied in \cite{KL07,KL16}. 
	Let $(X,T)$ be a t.d.s. and  $\mathcal{U} = (U_1,\dots,U_k)$ a subsets of $X$, we say that a subset $F \subseteq \mathbb{N}$ is an \emph{independence set} for $\mathcal{U}$ if for any non-empty finite subset $J \subseteq F$ and any $s = (s(j): j \in J) \in \{1,\dots,k\}^J$, we have $
	\bigcap_{j \in J} T^{-j} U_{s(j)} \neq \emptyset.$
	For a family $\F$, we call a tuple $\mathbf{ x}=(x_1,\ldots,x_k)$ is {\it independent tuple along $\F$} if for every product neighborhood $U_1\times \dots\times U_k$ of $\mathbf{x}$ the tuple $(U_1,\dots,U_k)$ has an independence set in $\F$.

	Let $\mathcal{A}:=\{p_1,p_2,\cdots,p_n,\dots\}$ be a family of distinct polynomials vanishing at $0$, for $d\in\N,n\in\Z$ we define the set 
	$\{p_{1}(n), p_{2}(n), \ldots, p_{d}(n)\}$ as a \emph{$\A_d$-set}. 
	Let $\mathcal{F}_{d\text{-}\A}$ denote the family of sets that contain infinitely many $\A_d$-sets. 
	Furthermore, let $\mathcal{F}_{\infty\text{-}\A}$ 
	denote the family of sets that contain at least one $\A_d$-set for each $d\in\mathbb{N}$. For $k\geq 2$ and a tuple $\mathbf{ x}=(x_1,\ldots , x_k) \in X^k$, $d\in \mathbb{N}\cup \{\infty\}$,
	we shall denote  the set of all independent $k$-tuples along $\mathcal{F}_{d\text{-} \A}$ 
	by $\mathrm{Ind}_{d \text{-} \A}^{(k)}(X)$.
	
	As an appplication of Theorem \ref{main thm}, we can character the complexity of most fibers in minimal systems along $\infty$-step nilfactor.

	\begin{thm}\label{main2}
		Let $\mathcal{A}:=\{p_1,p_2,\cdots,p_n,\dots\}$ be a family of distinct polynomials vanishing at $0$ and let $(X,T)$ be a minimal system.
		Then there exists a dense $G_\delta$ subset $\Omega$ of $X$ such that for any $x\in \Omega$ and any $k\ge 2$,
		\begin{equation*}
			(\mathbf{RP}^{[\infty]}[x])^{(k)}\subseteq Ind_{\infty \text{-} \A}^{(k)}(X),
		\end{equation*}
		where $\mathbf{RP}^{[\infty]}[x]=\{x'\in X:(x,x')\in \mathbf{RP}^{[\infty]}(X)\}$.
	\end{thm}

	\noindent{\bf Organization of the paper.} We organize the paper as follows. In Section 1, we give the motivation of this paper and state the main results. In Section 2 we introduce some necessary notions and some known facts to be used in the paper. In Section 3, we prove Theorem \ref{main thm} and Theorem \ref{main thm'}. In Section 4, we apply Theorem \ref{main thm} to demonstrate a proof of Theorem \ref{main2}.

	\medskip
	
	\noindent{\bf Acknowledgements.} The authors would like to thank Professors Xiangdong Ye and Jiahao Qiu for useful discussions. This research is supported by NSFC Tianyuan Youth Grant (12526551) and Shanghai Sailing Program (24YF2711500).

	\section{preparation}
	In this section we gather definitions and preliminary results that will be necessary later on.
	
	\subsection{Topological dynamical systems}
	A \emph{topological dynamical system} (t.d.s. for short) is a pair $(X,T)$, where $X$ is a compact metric space with a metric $\rho$ and $T \colon X \to X$ is a homeomorphism. For $x \in X$, $\mathcal{O}(x,T) = \{T^n x : n \in \mathbb{Z}\}$ denotes the \emph{orbit} of $x$. A t.d.s. $(X,T)$ is called \emph{minimal} if every point has dense orbit in $X$. A \emph{homomorphism} between t.d.s.\ $(X,T)$ and $(Y,T)$ is a continuous onto map $\pi \colon X \to Y$ which intertwines the actions; one says that $(Y,T)$ is a \emph{factor} of $(X,T)$, and that $(X,T)$ is an \emph{extension} of $(Y,T)$. One also refers to $\pi$ as a \emph{factor map} or an \emph{extension} and one uses the notation $\pi \colon (X,T) \to (Y,T)$. Then $R_\pi = \{(x_1, x_2) : \pi(x_1) = \pi(x_2)\}$ is a closed invariant equivalence relation, and $Y = X/R_\pi$.
	
	Let $\pi : (X, T) \to (Y, T)$ be a factor map.  One says that $\pi$ is an \textit{open} extension if it is open as a map; is a \textit{proximal} extension if $\pi(x_1) = \pi(x_2)$ implies $(x_1, x_2) \in \mathbf{P}(X, T)$; is a {\it distal} extension if $\pi(x_1) = \pi(x_2)$ and $x_1 \neq x_2$ implies $(x_1, x_2) \notin \mathbf{P}(X/T)$; is an \textit{almost one to one} extension if there exists a dense $G_\delta$ set $X_0 \subseteq X$ such that $\pi^{-1}(\{\pi(x)\}) = \{x\}$ for any $x \in X_0$. 
	
\begin{thm}\cite[Theorem 3.1]{Vee}\label{veech}
	Given a factor map $\pi: X \to Y$ of minimal systems $(X,T)$ and $(Y,S)$, there exists a commutative diagram of factor maps (called O-diagram)
	\[
	\begin{CD}
		X & @<\sigma<< & X^* \\
		@V\pi VV & & @V\pi^* VV \\
		Y & @<\tau<< & Y^*
	\end{CD}
	\]
	such that
	\begin{enumerate}
		\item[(a)] $\sigma$ and $\tau$ are almost one to one extensions;
		\item[(b)] $\pi^*$ is an open extension.
	\end{enumerate}
	We note that this diagram is canonical. In particular, when the map $\pi$ is open, we have $X^* = X$.
\end{thm}

\subsection{Nilmanifolds and nilsystems.} Let $G$ be a group. For $g, h \in G$ and $A, B \subset G$, we write $[g, h] = ghg^{-1}h^{-1}$ for the commutator of $g$ and $h$ and $[A, B]$ for the subgroup spanned by $\{[a, b] : a \in A, b \in B\}$. The commutator subgroups $G_j$, $j \ge 1$, are defined inductively by setting $G_1 = G$ and $G_{j+1} = [G_j, G]$. Let $d \ge 1$ be an integer. We say that $G$ is \textit{$d$-step nilpotent} if $G_{d+1}$ is the trivial subgroup.

Let $G$ be a $d$-step nilpotent Lie group and $\Gamma$ be a discrete cocompact subgroup of $G$. The compact manifold $X = G/\Gamma$ is called a \textit{$d$-step nilmanifold}. The group $G$ acts on $X$ by left translations and we write this action as $(g, x) \mapsto gx$. The Haar measure $\mu$ of $X$ is the unique probability measure on $X$ invariant under this action. Let $\tau \in G$ and $T$ be the transformation $x \mapsto \tau x$ of $X$. Then $(X, \mu, T)$ is called a \textit{$d$-step nilsystem}. In the topological setting we omit the measure and just say that $(X, T)$ is a $d$-step nilsystem.
In the topological setting we omit the measure and just say that $(X, T)$ is a $d$-step nilsystem. A $d$-step pro-nilsystem is an inverse limit of $d$-step nilsystems.
	
	\subsection{Regionally proximal relation of order $\infty$}
		Let $(X,T)$ be a t.d.s.. A pair $(x,y) \in X \times X$ is \emph{proximal} if
	\[
	\inf_{n \in \mathbb{Z}} \rho(T^n x, T^n y) = 0
	\]
	and \emph{distal} if it is not proximal. Denote by $\mathbf{P}(X)$ the set of all proximal pairs of $X$. The t.d.s. $(X,T)$ is \emph{distal} if $(x,y)$ is a distal pair whenever $x,y \in X$ are distinct.
	
	\begin{defn}
		Let $(X,T)$ be a t.d.s. and $d \in \mathbb{N}$. The \emph{regionally proximal relation of order $d$} is the relation $\mathbf{RP}^{[d]}$($\RP_X^{[d]}$ in case of ambiguity) defined by: $(x,y) \in \mathbf{RP}^{[d]}$ if and only if for every $\delta > 0$, there exist $x',y' \in X$ and $\vec{n} \in \mathbb{N}^d$ such that $
		\rho(x,x') < \delta, \rho(y,y') < \delta$, 
		and
		\[
		\rho(T^{\vec{n} \cdot \varepsilon} x', T^{\vec{n} \cdot \varepsilon} y') < \delta, \  \forall \varepsilon \in \{0,1\}^d \setminus \{\vec{0}\}.
		\]
		For any $x\in X$, denote $\mathbf{RP}^{[\infty]}[x]=\{x'\in X:(x,x')\in \mathbf{RP}^{[\infty]}\}$. 
	\end{defn}

	It is known in \cite{SY} that for any minimal system $(X,T)$ and $d\in\N$, $\RP^{[d]}(X)$ is a closed invariant equivalence relation, and  $(X,T)$ is a \emph{$d$-step pro-nilsystem} if and only if $\mathbf{RP}^{[d]}$ is trivial\cite{SY}. The quotient of $(X,T)$ under $\RP_X^{[d]}$ is the \emph{maximal $d$-step pro-nilfactor} of $X$.  
	For any minimal t.d.s. $(X,T)$,  $\RP^{[\infty]}=\bigcap_{d\geq 1}\RP^{[d]}$ is also a closed invariant equivalence relation. 
	\begin{defn}
		A minimal system $(X,T)$ is an \emph{$\infty$-step pro-nilsystem}, if the equivalence relation $\mathbf{RP}^{[\infty]}$ is trivial, i.e., coincides with the diagonal. We say the quotient under $\mathbf{RP}^{[\infty]}$ is the \emph{maximal $\infty$-step pro-nilfactor} of $X$.
	\end{defn}
	
Let $\pi:(X,T)\ra(Y,T)$, then $(Y,T)$ is a $\infty$-step pro-nilsystem if and only if $R_\pi\subseteq \RP^{[\infty]}$. The following structure theorems characterize the $\infty$-step pro-nilsystem. 
	
	\begin{thm}\cite{DDMSY}
		A minimal system is an $\infty$-step pro-nilsystem if and only if it is an inverse limit of minimal nilsystems.
	\end{thm}
	
	\subsection{Families}
	
	Recall that a collection $\mathcal{F}$ of subsets of $\mathbb{Z}$ is a \emph{family}
	if it is hereditary upward, i.e.,
	$F_1 \subseteq F_2$ and $F_1 \in \mathcal{F}$ imply $F_2 \in \mathcal{F}$.
	A family $\mathcal{F}$ is called \emph{proper} if it is neither empty nor the entire power set of $\mathbb{Z}$,
	or, equivalently if $\mathbb{Z}\in \mathcal{F}$ and $\emptyset\not\in \mathcal{F}$.
	For a family $\mathcal{F}$ its \emph{dual} is the family
	$\mathcal{F}^*=\{   F\subseteq \mathbb{Z}: F\cap F' \neq \emptyset \;\mathrm{for} \; \mathrm{all}\; F'\in \mathcal{F} \} $.
	It is not hard to see that
	$\mathcal{F}^*=\{F\subseteq \mathbb{Z}:\mathbb{Z}\backslash F\notin \mathcal{F}\}$, from which we have that if $\mathcal{F}$ is a family then $(\mathcal{F}^*)^*=\mathcal{F}$.
	If a family $\mathcal{F}$ is closed under finite intersections and is proper, then it is called a \emph{filter}.
	A family $\mathcal{F}$ has the {\it Ramsey property} if $A = A_1\cup A_2 \in \mathcal{F}$ implies that $A_1 \in \mathcal{F}$
	or $A_2 \in F$. It is well known that a proper family has the Ramsey property if and
	only if its dual $\mathcal{F}^*$ is a filter \cite{FH}.

	For $d\in \mathbb{N}$ and a finite subset $\{c_1, \ldots, c_d\}$ of $\mathbb{Z}$, the
	\emph{finite IP-set of length $d$} (IP$_d$-set for short) generated by $\{c_{1}, \ldots, c_d\}$
	is the set
	\[
	\big\{c_{1}\epsilon_{1}+ \cdots+ c_d\epsilon_d: \epsilon_1,\ldots,\epsilon_d\in \{0,1\}\big\} \backslash \{0\}.
	\]
	The collection of all sets containing finite IP-sets with arbitrarily long lengths is denoted by $\mathcal{F}_{fip}$. It follows from \cite[Lemma 8.1.6]{HSY16} that $\mathcal{F}_{fip}$ has the Ramsey property. i.e., $\mathcal{F}_{fip}^{*}$ is a filter. The next lemma is clear.
	\begin{lem}\label{fip*}
		For any $B_1,\ldots, B_{k}\in \mathcal{F}_{fip}^{*}$ and $A\in \mathcal{F}_{fip}$ we have $A\cap B_1\cap \cdots\cap B_k \neq\emptyset$. 
	\end{lem}

	\begin{lem}\cite[Thoerem 4.3]{BL03}\label{pol rec for open set}
		Let $(X,T)$ be a minimal system. Let $p_1,p_2,\cdots,p_k$ be distinct integral polynomials vanishing at $0$. For any $k\in \mathbb{N}$ and any non-empty open subset $U$ of $X$, we have
		\[ 
		\{ n\in \mathbb{Z}: U\cap T^{-p_1(n)}U\cap\cdots\cap T^{-p_k(n)}U\neq\emptyset\}\in \mathcal{F}_{fip}^{*}.
		\]
	\end{lem}

	\subsection{Independence}
	
	The notion of \emph{independence} was firstly introduced and studied in \cite{KL07,KL16}.
	It corresponds to a modification of the notion of \emph{interpolating set}
	studied in \cite{GW95,HY06}.

	\begin{defn}
		Let $(X,T)$ be a t.d.s. and $k\geq 2$. Given a tuple $\mathcal{U} = (U_1,\ldots,U_k )$ of subsets
		of $X$ we say that a subset $F\subseteq \mathbb{N}$  is an \emph{independence set} for $\mathcal{U}$ if for any non-empty
		finite subset $ J\subseteq  F$ and any $s=(s(j):j\in J)\in \{1,\ldots,k\}^J$  we have
		\[
		\bigcap_{j\in J}T^{-j}U_{s(j)}\neq \emptyset.
		\]
		
		We shall denote the collection of all independence sets for $\mathcal{U}$
		by $\mathrm{Ind}(U_1 ,\ldots,U_k )$ or $\mathrm{Ind}(\mathcal{U})$.
	\end{defn}
	
	\begin{defn}
		Let $\mathcal{F}$ be a family, $k\geq 2$ and $(X,T)$ be a t.d.s..
		A tuple $(x_1,\ldots,x_k)\in X^k$ is called an {\it independent tuple along $\mathcal{F}$}
		if for any neighborhoods  $U_1,\ldots,U_k$ of $x_1,\ldots,x_k$ respectively, one has $\mathrm{Ind}(U_1,\ldots,U_k)\cap\mathcal{F}\neq\emptyset$.
	\end{defn}

	The following proposition can be verified by definition straightforward.

	\begin{prop}\label{factor-ind}
		Let $\pi:(X,T)\to (Y,T)$ be a factor map of t.d.s., $k\ge 2$ and $\mathcal{F}$ be a family of $\mathbb{N}$.
		If $(x_1,\ldots,x_k )\in X^{k}$ is an independent tuple along $\mathcal{F}$,
		then  $(\pi(x_1),\ldots,\pi(x_k) ) $ is also an independent tuple along $\mathcal{F}$.
	\end{prop}
	
\subsection{Saturation theorems}
\subsubsection{Linear cases}
Given a map $\pi: X \to Y$ of sets $X$ and $Y$, a subset $L$ of $X$ is called \emph{$\pi$-saturated} if $L = \pi^{-1}(\pi(L))$. Let $\pi: (X,T) \to (Y,T)$ be a factor map of t.d.s. and $d \in \mathbb{N}$. $(Y,T)$ is said to be a \emph{$d$-step topological characteristic factor }(TCF for short) if there exists a dense $G_\delta$ subset $\Omega$ of $X$ such that for each $x \in \Omega$ the orbit closure
\[
L_x^d(X) = \overline{\mathcal{O}}\big((x^{(d)}, T \times \cdots \times T^d\big)
\]
is $\pi^{(d)}$-saturated, where $\pi^{(d)} = \pi \times \cdots \times \pi$ ($d$-times) and $x^{(d)} = (x,\dots,x)$ ($d$-times). That is, $(x_1,\dots,x_d) \in L_x^d(X)$ if and only if $(x'_1,\dots,x'_d) \in L_x^d(X)$ whenever $\pi(x_i) = \pi(x'_i)$ for $i=1,\dots,d$. 
	
	\begin{thm}\cite[Theorem A]{GHSWY}\label{ghswy}
		Let $(X, T)$ be a minimal system and let $\pi : X \to X/\mathbf{RP}^{[\infty]} = X_\infty$ be the factor map. Then there exist minimal systems $X^*$ and $X_\infty^*$ which are almost one to one extensions of $X$ and $X_\infty$ respectively, and a commuting diagram below such that $X_\infty^*$ is a $d$-step TCF of $X^*$ for all $d \ge 2$.
		\[
		\begin{CD}
			X & @<\sigma<< & X^* \\
			@V\pi VV & & @V\pi^* VV \\
			X_\infty & @<\tau<< & X_\infty^*
		\end{CD}
		\]
	\end{thm}

In \cite{QXYY25}, Qiu, Xu, Ye and Yu defined the notion of weak topological characteristic factor (wTCF for short).  
\begin{defn}\label{de wTCF}
	Let $\pi:(X,T)\to (Y,T)$ be a factor map of  t.d.s.  and $d\in\mathbb{N}$.
	We say that $(Y,T)$ is a {\it $d$-step wTCF of $(X,T)$}
	if there exists a dense $G_\delta$ subset $\Omega$ of $X$ such that 
	$(\pi^{-1}(\pi(x))^d \subseteq L_x^d(X)$ for any $x\in \Omega$. Further, we say $(Y,T)$ is an {\it $d$-step wTCF of $(X,T)$} if it is a $d$-step wTCF for every $d\in\N$.
\end{defn}

Now Theorem \ref{ghswy} can be expressed as: for any minimal system $(X,T)$ and $d\in\mathbb{N}$, $(X_\infty,T)$ is a $d$-step wTCF of $(X,T)$.

\medskip
The authors in \cite{QXYY25} also gave a saturation theorem for product systems. 

\begin{thm}\cite[Theorem A]{QXYY25}
	Let $k\in\mathbb{N}$ and $(X_i,T_i),1\leq i\leq k$ be minimal systems.
	For $d\in \mathbb{N}\cup \{\infty\}$, let
	$X_{i,d}$ be the maximal $d$-step pro-nilfactor of $X_i$ respectively. Then 
	$(\prod_{i=1}^{k}X_{i,d},\prod_{i=1}^{k}T_i)$ is a $d$-step wTCF of $(\prod_{i=1}^{k}X_{i},\prod_{i=1}^{k}T_i)$.
\end{thm}

\subsubsection{Polynomial extensions}
	The notion of TCF of order $d$ can be extended to general polynomials. 
	Note that we say polynomials $p$ and $q$ are {\it distinct} if $p-q$ is not constant (it is called {\it essentially distinct} in \cite{Le05A}). Let $C=\{p_1,\ldots, p_d\}$ be a set of distinct non-constant integral polynomials vanishing at $0$. 
	Given a factor map $\pi: (X,T)\rightarrow (Y,T)$ and $d\ge 2$,
	the t.d.s. $(Y,T)$ is said to be a {\em $d$-step TCF for $C$ 
		of $(X,T)$}, if there exists a dense $G_\d$
	subset $\Omega$ of $X$ such that for each $x\in \Omega$ the orbit
	closure $L_x^{C}=\overline{\{(T^{p_1(n)}x, \ldots, T^{p_d(n)}x): n\in\Z\}}$ is $\pi^{(d)}=\pi\times \cdots \times
	\pi$ ($d$-times) saturated.
	
	For polynomials, we have the following theorem, which is a polynomial version of Theorem \ref{ghswy}.

	\begin{thm}\cite[Theorem 3.7]{HSY25}\label{hsy}
		 Let $(X,T)$ be a minimal t.d.s., and $\pi: X \to X_\infty$ be the factor map from $X$ to its maximal $\infty$-step pro-nilfactor $X_\infty$ of $X$. Then there are minimal t.d.s. $X^*$ and $X_\infty^*$ which are almost one to one extensions of $X$ and $X_\infty$ respectively, an open factor map $\pi^*$ and a commuting diagram below
		\[
		\begin{CD}
			X & @<\sigma<< & X^* \\
			@V\pi VV & & @V\pi^* VV \\
			X_\infty & @<\tau<< & X_\infty^*
		\end{CD}
		\]
		such that there is a $T$-invariant dense $G_\delta$ subset $X_0^*$ of $X^*$ having the following property: for all $x \in X_0^*$, for any non-empty open subsets $V_1,\dots,V_d$ of $X^*$ with $\pi(x) \in \bigcap_{i=1}^d \pi^*(V_i)$ and distinct non-constant integral polynomials $p_1,p_2,\dots,p_d$ with $p_i(0)=0$, $i=1,2,\dots,d$, there is some $n \in \mathbb{N}$ such that
		\[
		x \in T^{-p_1(n)}V_1 \cap T^{-p_2(n)}V_2 \cap \dots \cap T^{-p_d(n)}V_d.
		\]
	\end{thm}

	It is shown in \cite{YY25} that Theorem \ref{hsy} has the following equivalent statement.
	\begin{thm}\label{polsat}
	Let $(X,T)$ be a minimal system, and $\pi:X\rightarrow X_\infty$ be the factor map, where $X_\infty$ is the maximal $\infty$-step
	pro-nilfactor of $X$. Let $d\in\N$ and $p_1, p_2,\ldots, p_d$ be distinct non-constant integral polynomials vanishing at $0$. Then there is a dense $G_\delta$ set $\Omega$ of $X$ such that for any 
	$x\in \Omega$,
	$$(\pi^{(d)})^{-1}\pi^{(d)}(x^{(d)})\subset L_x^{C}=:\overline{\{(T^{p_1(n)}x, \ldots, T^{p_d(n)}x): n\in\Z\}},$$
	where $C=\{p_1,\ldots,p_d\}$.
	Particularly, if $\pi$ is open then $X_\infty$ is the $d$-step TCF  for C, 
	i.e., $(\pi^{(d)})^{-1}\pi^{(d)}(L_x^{C})=L_x^{C}.$
\end{thm}

\begin{defn}
	Let $\pi:(X,T)\to (Y,T)$ be a factor map of  t.d.s.  and $d\in\mathbb{N}$. Let $C=\{p_1,\ldots, p_d\}$ be a set of distinct non-constant integral polynomials vanishing at $0$.
	We say that $(Y,T)$ is a {\it $d$-step wTCF for $C$ of $(X,T)$} 
	if there exists a dense $G_\delta$ subset $\Omega$ of $X$ such that 
	$(\pi^{-1}(\pi(x))^d \subseteq L_x^C(X)$ for any $x\in \Omega$. 
\end{defn}

  By similar proof in \cite{QXYY25}, we have the following proposition which gives an equivalent conditions for wTCF for $C$.
\begin{prop}\label{equi of wtcf}
	Let $\pi:(X,T)\to (Y,T)$ be a factor map of  t.d.s., $C=\{p_1, p_2,\ldots, p_d\}$ be distinct non-constant integral polynomials vanishing at $0$ and $d\in\mathbb{N}$. Then the following are equivalent:
	\begin{enumerate}
		\item For any open subsets $V_0,V_1,\ldots,V_d$ of $X$ with
		$\bigcap_{i=0}^d\pi(V_i)$ having non-empty interior, there exists $n\in\mathbb{N}$
		such that $V_0\cap T^{-p_1(n)}V_1\cap\cdots\cap T^{-p_d(n)}V_d\neq\emptyset$.
		\item $Y$ is a $d$-step wTCF for $C$ of $X$.
	\end{enumerate}
\end{prop}
\begin{proof}
(2)$\Rightarrow$(1) is follows from the definition.

(1)$\Rightarrow$(2). Let $F : X \to 2^{X^d}$ be the map defined by $x \mapsto L_x^C$. Then $F$ is l.s.c. and  the continuous points of $F_d$ form a residual subset $\Lambda_1$ of $X$. Let $G : Y \to 2^X$ be the map defined by $y \mapsto \pi^{-1}(y)$. Then $G$ is u.s.c. and the continuous points of $G$ form a residual subset $\Lambda_2$ of $Y$.

Set $\Lambda = \Lambda_1 \cap \pi^{-1}(\Lambda_2)$. Then $\Lambda$ is a dense $G_\delta$ subset of $X$. We claim $(\pi^{-1}(\pi(x)))^d \subseteq L_x^C(X)$ for every $x \in \Lambda$, which implies $Y$ is a $d$-step wTCF for $C$ of $X$.

Let $x \in \Lambda$ and $\pi(x_1)=\pi(x_2)= \dots=\pi(x_d)$. Fix $\delta > 0$. As $F$ is continuous at $x$, there is some $\eta_1 > 0$ with $\eta_1 < \delta$ such that whenever $w \in B_{\eta_1}(x)$ one has
\[
\rho_H^d\left(L_w^C, L_x^C \right) < \delta,
\]
where $\rho_H^d$ is the Hausdorff metric on $2^{X^d}$. 
As $G$ is continuous at $\pi(x)$, there is some $\eta_2 > 0$ with $\eta_2 < \delta$ such that whenever $y \in B_{\eta_2}(x)$ one has
\[
\rho_H\left(\pi^{-1}(\pi(y)), \pi^{-1}(\pi(x))\right) < \delta,
\]
where $\rho_H$ is the Hausdorff metric on $2^X$.

Set $\eta = \min\{\eta_1, \eta_2\}$, $\pi(x)$ is an interior point of sets $\pi(B_\eta(x))$, $\pi(B_\delta(x_j)),1\leq j\leq d$, thus by the assumption, there exists $n \in \mathbb{N}$ and $z \in X$ such that
\[
z \in B_\eta(x) \cap T^{-p_1(n)}B_\delta(x_1) \cap \cdots \cap T^{-p_d(n)}B_\delta(x_d).
\]

So we get that
\[
\rho^d\left((T^{p_1(n)} z,\dots,T^{p_d(n)} z), L_x^C\right) < \delta,
\]
and thus
\[
\rho^d\left((x_1, \dots, x_d), L_x^C\right) < 2\delta,
\]
where $\rho^d$ is the metric on $X^d$. As $\delta$ is arbitrary, we conclude that $(x_1, \dots, x_d) \in L_x^C$.

\end{proof}

\section{The proof of Theorem \ref{main thm} and Theorem \ref{main thm'}}
	
	Denote $[k]:=\{1,2,\dots,k\}$. First, we have the following proposition which is the product extension of Proposition \ref{equi of wtcf}.
	\begin{prop}\label{wTCF}
		Let $\pi_i:(X_i,T_i)\to (Y_i,T_i),i\in[k]\}$ be factor maps of t.d.s. and $C=\{p_1, p_2,\ldots, p_d\}$ be distinct non-constant integral polynomials vanishing at $0$. Then the following are equivalent.
		\begin{enumerate}
			\item For any open subsets $V^{(i)}_0,V^{(i)}_1,\ldots,V^{(i)}_d$ of $X$ with
			$\bigcap_{i=0}^d\pi_i(V^{(i)}_i)$ having non-empty interior, there exists $n\in\mathbb{N}$
			such that $V_0^{(i)}\cap T^{-p_1(n)}V^{(i)}_1\cap\cdots\cap T^{-p_d(n)}V^{(i)}_d\neq\emptyset$.
			\item $\prod_{i=1}^{k}Y_i$ is a $d$-step wTCF for $C$ of $\prod_{i=1}^{k}X_i$.
		\end{enumerate}
	\end{prop}

	We can assume that $\pi$ is open. For the general case, we can use O-diagram to prove the conclusion.  
    
    Now we first prove the open case of Theorem \ref{main thm}.
	
	\begin{lem}\label{keyopen}
		For $i\in [k]$, let $\pi_{i}: (X_i,T_i)\rightarrow (Y_i,T_i)$ be a factor map of minimal systems, and $p_1, p_2, \ldots, p_d$ be distinct non-constant integral polynomials vanishing at $0$.
		For every $i\in [k]$,
		if $\pi$ is open and $R_{\pi_{i}}\subseteq \mathbf{RP}_{X_i}^{[\infty]}$,
		then for any $d\in \mathbb{N}$ and any open subsets $V_{0}^{(i)},V_{1}^{(i)},\ldots, V_d^{(i)}$ of $X_i$ 
		with $\bigcap_{j=0}^d\pi_{i}(V_{j}^{(i)})\neq\emptyset$,
		there exists $n\in\mathbb{N}$ such that
		\[
		V_{0}^{(i)}\cap T_i^{-p_1(n)} V_{1}^{(i)}\cap \cdots \cap T_i^{-p_d(n)} V_d^{(i)}\neq\emptyset,\ \ \ \forall \; i\in [k].
		\]
	\end{lem}
	To prove Lemma \ref{keyopen}, we need the following lemma.
	\begin{lem}\label{key}
		Let $\pi:(X,T)\rightarrow (X_{\infty},T)$ be a factor map of minimal systems, $k\in \N$ and  let $p_1, p_2, \ldots, p_k$ be distinct non-constant integral polynomials vanishing at $0$. Then there is a dense $G_\delta$ set $\Omega$ of $X$ such that for any open subsets $V_0, V_1,\cdots, V_k$ with
		$\bigcap_{j=0}^k\pi(V_{j})\neq\emptyset$ and any $z\in V_0\cap \Omega$ with $\pi(z)\in \bigcap_{j=0}^k\pi(V_{j})$, there exists some $A\in \mathcal{F}_{fip}$ such that $T^{p_i(n)}z\in V_i$ for any $i\in[k]$ and $n\in A$.
	\end{lem}
	\begin{proof}
	Suppose $p_1, p_2, \ldots, p_k$ are distinct non-constant integral polynomials vanishing at $0$. 
	Let $d=\max_{1\leq i\leq k}\{\deg p_i\}$ and let $a_i$ be the coefficient of polynomial $p_i,\forall i\in[k]$. We may suppose $a_i>0$ for any $i\in [k]$. Choose $N\geq\max_{i\in [k]}\{a_i\}$ sufficiently large such that $$(1+1/N)^{d}<\min\{a_i/a_j:a_j<a_i,i,j\in [k]\}:=\omega.$$ Let $\{b_j\}_{j\in\N}$ be a sequence of positive integers such that $b_{j+1}\geq N\cdot(b_1+b_2\cdots+ b_j)+1$, and let $I_j$ be the finite IP-set generated by $\{b_1,\cdots,b_j\}$.
		\begin{claim}
			Let $p_{i,m}(n):=p_i(mn),\forall n\in\Z$. For $i,i'\in[k]$ and $m,m'\in I_j$, $p_{i,m}=p_{i',m'}$ if and only if $i=i'$ and $m=m'$.
		\end{claim}
		\begin{proof}
			Suppose $\deg p_i =\deg p_{i,m}:=d_i$, then the leading coefficient of $p_{i,m}$ is $a_i m^{d_i}$. We suppose for a contradiction that there exist $i,i'\in[k]$ with $i<i'$ and  $m,m'\in I_j$ such that $p_{i,m}=p_{i',m'}$. 
			
			Let $\vec{\ep}=(\epsilon_1,\cdots,\epsilon_j),\vec{\ep'}=(\epsilon_1',\cdots,\epsilon_j')\in \{0,1\}\backslash\{\vec{0}\}$ such that $m=\sum_{i=1}^j b_i\epsilon_i,m'=\sum_{i=1}^j b_i\epsilon'_i$. 
			When $p_{i,m}=p_{i',m'}$, they have the same degree and coefficient. i.e., we have $\deg p_{i,m}=\deg p_{i',m'}:=d^*$ and  $a_im^{d^*}=a_{i'}(m')^{d^*}$. 
			Let  $$t=\max\{t\in[j]:\epsilon_n=1\}\text{ and }t'=\max\{t\in[j]:\epsilon_n=1\}.$$ It is clear that $t\leq t'$ whenever $m<m'$.
			
			If $t=t'$, then  $$\frac{a_i}{a_{i'}}=(\frac{m'}{m})^{d^*}<(1+\frac{b_{t-1}+\cdots+b_1}{b_{t}})^{d^*}<(1+\frac{1}{N})^{d^*}<\omega,$$ we can get that $(\frac{m'}{m})^{d^*}<\frac{a_i}{a_{i'}}$, which is a contradiction by the definition of $\omega$.
			
			If $t<t'$, then $$\frac{a_i}{a_{i'}}=(\frac{m'}{m})^{d^*}>\dfrac{b_{t'>t}}{b_1+\dots+b_{t}}>N^{d^*}>\frac{a_i}{a_{i'}},$$ a contradiction. This shows the claim.
		\end{proof}
		
		Let  $C_j=\{p_{i,m}(n):i\in[k], m\in I_j\}$. By the claim above, we have that $C_j$ is a family of distinct integral polynomials vanishing at 0. Then by Theorem \ref{polsat}, there exists a dense $G_\delta$ set $\Omega_j$ such that for any 
		$x\in \Omega_j$,
		$$(\pi^{(k(2^j-1))})^{-1}\pi^{(k(2^j-1))}(x^{(k(2^j-1))})\subset L_x^{C_j}=:\overline{\{(T^{p_{i,m}(n)}x)_{m\in I_j,i\in[k]}: n\in\Z\}}.$$
			
		Let $\Omega=\cap_{j\in\N}\Omega_j$, it is also a dense $G_\delta$ set. Let $V_0,V_1,\cdots, V_k$ be open subsets of $X$ with $W:= \bigcap_{j=0}^k\pi V_{j}\ne\varnothing$, then $\pi^{-1}(W)\cap V_i\neq \emptyset$ for every $i\in[k]$. Let $z\in \Omega\cap V_0\cap\pi^{-1}(W)$. For $i\in[k]$, let $z_i\in \pi^{-1}(\{\pi(z)\})\cap V_i$ and choose $\delta>0$ such that $B(z_i,\delta)\subset V_i$.
			
		For every $j\in\N$. Let $$\overrightarrow{Z_j}=(z_1,z_2,\cdots,z_k,z_1,z_2,\cdots,z_k,\cdots,z_1,z_2,\cdots,z_k)\in X^{k(2^j-1)}.$$ 
		It is clear $\overrightarrow{Z_j}\in L_z^{C_j}$, which implies that there is $n_j\in\N$ such that $$\rho(T^{p_{i,m}(n_j)}z,z_i)<\delta, \forall m\in I_j,i\in[k].$$ 
		Let $A_j=\{mn_j:m\in I_j\}$ and $A=\bigcup_{j\in \N}A_j$, then we have $A\in \mathcal{F}_{fip}$ and $T^{p_i(n)}z\in V_i$ for every $i\in[k]$ and $n\in A$. 
		
	\end{proof}
	Then we are ready to prove the key lemma.
	\begin{proof}[Proof of Lemma \ref{keyopen}]
		Let $d\in\mathbb{N}$ and assume that  $V_{0}^{(i)},V_{1}^{(i)},\ldots, V_d^{(i)}$ are open subsets of $X_i$
		with $W_i:=\bigcap_{j=0}^d\pi_{i}(V_{j}^{(i)})\neq\emptyset$ for every $i\in [k]$.
		We next construct inductively  $n_1,\ldots,n_k\in\mathbb{N}$ such that for any $i\in[k]$,
		\begin{itemize}
			\item $V_0^{(l)}\cap T_l^{-p_1(n_i)}V_1^{(l)}\cap\cdots\cap T_l^{-p_d(n_i)}V_d^{(l)}\neq\emptyset$ for $l\in\{1,\ldots,i\}$,
			\item $W_l\cap T_l^{-p_1(n_i)}W_l\cap\cdots\cap T_l^{-p_d(n_i)}W_l\neq\emptyset$ for $l\in\{i+1,\ldots,k\}$.
		\end{itemize}
		Assume this has been achieved, then
		\[
		V_0^{(i)}\cap T_i^{-p_1(n_k)}V_1^{(i)}\cap\cdots\cap T_i^{-p_d(n_k)}V_d^{(i)}\neq\emptyset, \ \ \ \forall \; i\in [k],
		\]
		as was to be shown.

		\medskip
		
		{\bf We now return to the inductive construction of $n_1,\ldots,n_k$.}
		
		\medskip

		\noindent {\bf Step $1$:}
		By Lemma \ref{key}, there exists $x_1\in V_{0}^{(1)}$ and $A_1\in \mathcal{F}_{fip}$ such that for each $n\in A_1$,
		\[
		(T_1^{p_1(n)}x_1, \ldots, T_1^{p_d(n)}x_1)\in V_{1}^{(1)}\times \cdots\times  V_d^{(1)}.
		\]
		Applying Lemma \ref{pol rec for open set} to systems $(Y_2, T),\ldots,(Y_{k},T)$ with open sets $W_2,\ldots,W_{k}$ respectively, we have
		\[
		\{ n\in \mathbb{Z}: W_l\cap T_l^{-p_1(n)}W_l\cap\cdots\cap T_l^{-p_d(n)}W_l\neq\emptyset\}\in\mathcal{F}_{fip}^{*},
		\]
		for each $l\in\{2,\ldots,k\}$. Thus, by Lemma \ref{fip*} there is $n_1\in A_1$ such that
		\begin{itemize}
			\item $V_0^{(1)}\cap T_1^{-p_1(n_1)}V_1^{(1)}\cap\cdots\cap T_1^{-p_d(n_1)}V_d^{(1)}\neq\emptyset$,
			\item $W_l\cap T_l^{-p_1(n_1)}W_l\cap\cdots\cap T_l^{-p_d(n_1)}W_l\neq\emptyset$ for $l\in\{2,\ldots,k\}$.
		\end{itemize}

		\medskip
		
		\noindent {\bf Step $i$:}
		Let $i\in \{2,\ldots,k\}$ and assume that we have already chosen $n_{i-1}\in\mathbb{N}$ such that
		\begin{itemize}
			\item $V_0^{(l)}\cap T_l^{-p_1(n_{i-1})}V_1^{(l)}\cap\cdots\cap T_l^{-p_d(n_{i-1})}V_d^{(l)}\neq\emptyset$ for $l\in\{1,\ldots,i-1\}$,
			\item $W_l\cap T_l^{-p_1(n_{i-1})}W_l\cap\cdots\cap T_l^{-p_d(n_{i-1})}W_l\neq\emptyset$ for $l\in\{i,\ldots,k\}$.
		\end{itemize}

		Note that
        $$\pi_j(V_{0}^{(j)})\cap \pi_j(T_j^{-p_1(n_{i-1})}V_1^{(j)})\cap \cdots\cap \pi_j(T_j^{-p_d(n_{i-1})}V_d^{(j)})\neq\emptyset, \forall j\in[i-1]$$
        and
		\begin{align*}
			 \pi_i(V_{0}^{(i)})\cap &\pi_i(T_i^{-p_1(n_{i-1})}V_1^{(i)})\cap \cdots\cap \pi_i(T_i^{-p_d(n_{i-1})}V_d^{(i)})  \\
			&\supseteq\;   W_i\cap T_i^{-p_1(n_{i-1})}W_i\cap\cdots\cap T_i^{-p_d(n_{i-1})}W_i\neq\emptyset.
		\end{align*}

        Now let $p_j^*$ be polynomials defined by $p_j^*(n)=p_j(n+n_{i-1})-p_j(n_{i-1})$ with $j\in[d]$. Since $p_1, p_2, \ldots, p_d$ are distinct the set $C^*=\{p_1^*,p_2^*,\cdots,p_d^*\}$ is also a family of distinct polynomials vanishing at 0. Since $\RP^{[\infty]}(X_j)$ is closed and $T_j\times T_j$-invariant, it is clear that $\prod_{j=1}^i X_{j,\infty}$ is the $\infty$-step pro-nilfactor of $\prod_{j=1}^i X_{j}$ where $X_{j,\infty}$ is the $\infty$-step pro-nilfactor of $X_{j}$. 
		By applying Lemma \ref{key} to the product system $(\prod_{j=1}^i X_j,\prod_{j=1}^{i}T_j)$ and open sets $\prod_{j=1}^{i} T_j^{-p_1(n_1)}V_1^{(j)},\dots,\prod_{j=1}^{i} T_j^{-p_1(n_1)}V_d^{(j)}$, there exists $x_j\in V_{0}^{(j)}$ and $A_i\in \mathcal{F}_{fip}$
		such that for each $n\in A_i$,
		\[ 
		(T_j^{p^*_1(n)}x_j, \ldots, T_j^{p^*_d(n)}x_j)\in T_j^{-p_1(n_{i-1})}V_1^{(j)}\times\cdots\times T_j^{-p_d(n_{i-1})}V_d^{(j)},\forall j\in[i],i\in[k].
		\]

		Now set
		$$U^{(l)}  =V_0^{(l)}\cap T_l^{-p_1(n_{i-1})}V_1^{(l)}\cap \cdots\cap T_l^{-p_d(n_{i-1})}V_d^{(l)}, \forall l\in\{1,\ldots,i-1\}, $$
		$$W^{(l)} =W_l\cap T_l^{-p_1(n_{i-1})}W_l\cap\cdots\cap T_l^{-p_d(n_{i-1})}W_l, \forall l\in\{i+1,\ldots,k\}.$$
		Applying Lemma \ref{pol rec for open set} to systems $(X_1, T),\ldots,(X_{i-1},T)$ and $(Y_{i+1}, T), \ldots,(Y_{k},T)$ for polynomials $p_1^*,\dots,p_d^*$ with open sets
		$U^{(1)},\ldots,U^{(i-1)}$ and $W^{(i+1)},\ldots,W^{(k)}$ respectively,
		by Lemma \ref{fip*} again there is some $m\in A_i$ such that
		\begin{itemize}
			\item $U^{(l)}\cap T_l^{-p^*_1(m)}U^{(l)}\cap\cdots\cap T_l^{-p^*_d(m)}U^{(l)}\neq\emptyset$ for $l\in\{1,\ldots,i-1\}$,
			\item $W^{(l)}\cap T_l^{-p^*_1(m)}W^{(l)}\cap\cdots\cap T_l^{-p^*_d(m)}W^{(l)}\neq\emptyset$ for $l\in\{i+1,\ldots,k\}$.
		\end{itemize}
		
		Set $n_i=n_{i-1}+m$.
		We get that
        $$
        \begin{aligned}
        x_j\in V_0^{(l)}&\cap T_j^{-p_1(n_{i-1})-p_1^*(m)}V_1^{(l)}\cap\dots \cap T_j^{-p_d(n_{i-1})-p_d^*(m)}V_d^{(l)}\\
        &=V_0^{(l)} \cap T_j^{-p_1(n_{i})}V_1^{(l)}\cap\dots \cap T_j^{-p_d(n_{i})}V_d^{(l)}\ne\emptyset, \forall j\in[i].
        \end{aligned}$$
        
        Therefore,
		\begin{itemize}
			\item $ V_0^{(l)}\cap T_l^{-p_1(n_i)}V_1^{(l)}\cap \cdots\cap T_l^{-p_d(n_i)}V_d^{(l)}\neq\emptyset$ for $l\in\{1,\ldots,i\}$,
			\item $W_l\cap T_l^{-p_1(n_i)}W_l\cap\cdots\cap T_l^{-p_d(n_i)}W_l\neq\emptyset$ for $l\in\{i+1,\ldots,k\}$.
		\end{itemize}
		
        We finish the construction by induction and hence ends the proof.
	\end{proof}

	Now we prove the main theorem.
	
	\begin{proof}[Proof of Theorem \ref{main thm}]
		By Theorem \ref{hsy} for every $i\in [k]$ there exist minimal systems $X_i^*$ and $X_{i,\infty}^*$ which are almost one to one
		extensions of $X_i$ and $X_{i,\infty}$ respectively, and a commuting diagram below such that $\pi_i^*$ is open and
		$X_{i,\infty}^*$ is a
		$\infty$-step TCF of $X_i^*$.
		\[
		\xymatrix{
			X_i \ar[d]_{\pi_i}  & X_i^* \ar[d]^{\pi_i^*}  \ar[l]_{\sigma_i} \\
			X_{i,\infty}   & X_{i,\infty}^*      \ar[l]_{\tau_i}
		}
		\]

	As each $\sigma_i$ is almost one-to-one, we can deduce that $X_{i,\infty}$ is the maximal $\infty$-step pro-nilfactor of $X_i^*$. For open subsets $V_{0}^{(i)},V_{1}^{(i)},\ldots, V_d^{(i)}$ of $X_i$ 
	with $\bigcap_{j=0}^d\pi_{i}(V_{j}^{(i)})\neq\emptyset$, we consider sets $U_j^{(i)}=\sigma_i^{-1}(V_j^{(i)})$ in $X_i^*$, by Lemma \ref{keyopen} we have
	\[
	\prod_{i=1}^{k}U_0^{(i)}\cap \widetilde{T}^{-p_1(n)}\prod_{i=1}^{k}U_1^{(i)}\cap \cdots\cap 
	\widetilde{T}^{-p_d(n)}\prod_{i=1}^{k}U_d^{(i)}\neq\emptyset,
	\]
	where $\widetilde{T}=T_1\times\cdots\times T_k$. Hence, 
	\[
	\prod_{i=1}^{k}V_0^{(i)}\cap \widetilde{T}^{-p_1(n)}\prod_{i=1}^{k}U_1^{(i)}\cap \cdots\cap 
	\widetilde{T}^{-p_d(n)}\prod_{i=1}^{k}V_d^{(i)}\neq\emptyset.
	\]
	i.e.,
	\[
	V_{0}^{(i)}\cap T_i^{-p_1(n)} V_{1}^{(i)}\cap \cdots \cap T_i^{-p_d(n)} V_d^{(i)}\neq\emptyset,\ \ \ \forall \; 1\leq i\leq k.
	\]
	\end{proof}

\begin{proof}[Proof of Theorem \ref{main thm'}]
	It follows from Theorem \ref{main thm} and Proposition \ref{wTCF} that $(\prod_{i=1}^{k}X_{i,\infty},T)$ is a $d$-step wTCF for $C$ of $(\prod_{i=1}^{k}X_{i},\prod_{i=1}^{k}T_i)$ for any $d\in\mathbb{N}$.
\end{proof}

Moreover, we have the following theorem.
  
\begin{thm}\label{generalization of main1}
Let  $C=\{p_1,p_2,\cdots,p_d\}$ be a family of distinct polynomials vanishing at $0$ and $\pi_{i,l_i}: (X_i,T_i)\rightarrow (X_{i,l_i},T_i)$ be factor maps of minimal systems for every $1\leq i\leq k$ where $X_{i,l_i}$ is the maximal $l$-step pro-nilfactor of $X_i$. If $(X_{i,l_i},T_i)$ be the $d$-step wTCF for $C$ of $(X_i,T_i)$, then $(\prod_{i=1}^{k}X_{i,l_i},\prod_{i=1}^{k}T_i)$ is a $d$-step wTCF for $C$ of $(\prod_{i=1}^{k}X_{i},\prod_{i=1}^{k}T_i)$.
\end{thm}

To prove this theorem, we need the following results.

\begin{thm}\cite[Theorem 3]{Le05C}\label{lei05C}
	Let $(X, \mathscr{X}, \mu, T)$ be an ergodic system. For any $r,b \in \mathbb{N}$ there exists $k \in \mathbb{N}$ such that for any system of distinct non-constant polynomials $p_1, \dots, p_r : \mathbb{Z}^d \to \mathbb{Z}$ of degree $\le b$ and any $f_1, \dots, f_r \in L^\infty$ with $\E(f_1|Z_{k-1})\equiv 0$ where $Z_k$ is the maximal $k$-step pro-nilfactor of $X$, one has
	$$\lim_{N\to\infty} \frac{1}{|\Phi_N|} \sum_{u\in\Phi_N} T^{p_1(u)} f_1 \cdots T^{p_r(u)} f_r = 0$$
	in $L^2(X)$ for any Følner sequence $\Phi_N$ in $\mathbb{Z}^d$.
\end{thm}

\begin{rem}
    Let $(X,T)$ be a minimal $\infty$-step pro-nilsystem and $X_k$ be the $d$-step wTCF for $C=\{p_1(n),\dots,p_d(n)\}$ of $X$. Then $(X,T)$ is uniquely ergodic and $Z_k(X)=X_k$ where $Z_k(X)$ denote the maximal $k$-step pro-nilfactor of $X$. So, the $G_\delta$ set $\Omega$ taken in the definition has a full measure and we have $f= \mathbb{E}(f \mid \mathcal{Z}_{\leq k})$ a.e. $x\in X$ where $f\in L^{\infty}(X)$.
    
    Write $f = (f - \mathbb{E}(f \mid \mathcal{Z}_{k-1})) + \mathbb{E}(f \mid \mathcal{Z}_{k-1})$. It is easy to see that $\mathbb{E}\big((f - \mathbb{E}(f \mid \mathcal{Z}_{k-1})) \mid Z_{k-1}\big) = 0$. Replacing $f_i$ by $\mathbb{E}(f_i \mid \mathcal{Z}_k)$ in Theorem \ref{lei05C}, we have
    $$\|\lim_{N\ra\infty}\dfrac{1}{N}\sum_{n=0}^{N-1}(\prod_{j=1}^dT^{p_j(n)}f_j-\prod_{j=1}^dT^{p_j(n)}\E(f_j|\mathcal{Z}_{k}))\|_{L^2}=0.$$
\end{rem}

\begin{thm}\cite[Theorem A]{BL96}\label{rec th}
	Let $(X, \mathscr{X}, \mu, T)$ be a measure preserving transformation and let $A \in \mathscr{X}$ be a set with positive measure, then for each $C = \{p_1, \dots, p_k\}$ of distinct non-constant integral polynomials vanishing at $0$,
	\[
	\liminf_{N\to\infty} \frac{1}{N} \sum_{n=0}^{N-1} \mu\big(A \cap T^{-p_1(n)}A \cap T^{-p_2(n)}A \cap \cdots \cap T^{-p_k(n)}A\big) > 0.
	\]
\end{thm}

\begin{prop}
	Let  $C=\{p_1,p_2,\cdots,p_d\}$ be a family of distinct polynomials vanishing at $0$. If $(X_{i,l_i},T_i)$ be the $d$-step wTCF for $C$ of $(X_i,T_i)$ for every $i\in\{1,2,\dots,k\}$ where $X_{i,l}$ is the maximal $l$-step pro-nilfactor of $X_i$, then $(\prod_{i=1}^{k}X_{i,l_i},\prod_{i=1}^{k}T_i)$ is a $d$-step wTCF for $C$ of $(\prod_{i=1}^{k}X_{i,\infty},\prod_{i=1}^{k}T_i)$ where $X_{i,\infty}$ is the maximal $\infty$-step pro-nilfactor of $X_i$.
\end{prop}
\begin{proof}
We show the case $k=2$ and, for general $k\in\N$, it is essentially the same. 
Let $\pi_1 : X_{i,\infty} \to X_{i,l_i},i=1,2$ be the factor maps.
	Set $\mu_i$ be the uniquely ergodic measure of $(X_{i,\infty}, T_i)$, and $(Z_{i,l_i}, \mathcal{Z}_{i,l_i}, \mu_{i,l_i}, T_i)$ be the maximal $l_i$-step pro-nilfactor of $(X_{i,\infty}, \mathcal{X}_{i,\infty}, \mu_i, T_i)$.
	It is known that $Z_{i,l_i} = X_{i,l_i}$, so $\pi_i$ can be viewed as the continuous factor map from $(X_{i,\infty}, \mu_i, T_i)$ to $(Z_{i,l_i}, \mu_{i,l_i}, T_i)$. In this case, $\mu_{1}\times\mu_{2}$ and $\mu_{1,l_1}\times\mu_{2,l_2}$ are the unique measure on $X_{1,\infty}\times X_{2,\infty}$ and $Z_{1,l_1}\times Z_{2,l_2}$ respectively.
	
	Let $\mu_{i,\infty} = \int_{Z_{i,l_i}} \mu_{i,z_i} d\mu_{i,l_i}(z_i)$ be the disintegration of $\mu_{i,\infty}$ over $\mu_{i,l_i}$. 
	For $d \in \mathbb{N}$, let $L_{d}^\mu = \text{Supp}(\lambda_d)$, where $$
	\lambda_d = \int_{Z_{1,l_1}\times Z_{2,l_2}} \prod_{j=1}^d \mu_{1,z_1}\times \mu_{2,z_2}\, d(\mu_{1,l_1}\times\mu_{2,l_2})(z_1,z_2).$$
	
	\noindent{\bf Claim 1.} $\text{Supp}(\lambda_d)=R_{\pi_1\times\pi_2}^d$ where $$R_{\pi_1\times\pi_2}^d:=\{(x^{(1)}_j,x^{(2)}_j)_{j=1}^d: \pi_i x^{(i)}_{1}=\dots \pi_i x^{(i)}_{d},i=1,2\}.$$
\begin{proof}
We note that $\lambda_d^k(R_{\pi_k}^d) = 1$, so $\text{Supp}(\lambda_d^k) \subset R_{\pi_k}^d$. There is a measurable set $Y_i \subset X_{i,l_i}$ with full measure such that for any $y_i \in Y_i$, $\text{Supp}(\mu_{i,y_i}) = \pi_i^{-1}(y_i), i=1,2$. Let $W = \text{Supp}(\lambda_d)$. Since
	\[
	\lambda_d(W) = \int_{Y_1\times Y_2} \prod_{j=1}^d \mu_{1,y_1}\times\mu_{2,y_2}(W) \, d(\mu_{1,l_1}\times\mu_{2,l_2})(y_1,y_2) = 1,
	\]
	we have that for a.e. $(y_1,y_2) \in Y_1\times Y_2$, $\prod_{j=1}^d \mu_{1,y_1}\times\mu_{2,y_2}(W) = 1$. This implies that $\prod_{j=1}^d \text{Supp}(\mu_{1,y_1}\times\mu_{2,y_2}) \subset W$, a.e. $(y_1,y_2) \in Y_1\times Y_2$. Thus by the distality of $\pi_{i}$, the map $(y_1,y_2) \mapsto (\pi_1\times\pi_2)^{-1}(y_1,y_2)$ (from $X_{1,l_1}\times X_{2,l_2}$ to $2^{X_{1,\infty}\times X_{2,\infty}}$) is continuous and we conclude that $R_{\pi_1\times \pi_2}^d \subset\text{Supp}(\lambda_d)$. Thus, we get that $\text{Supp}(\lambda_d) = R_{\pi_1\times \pi_2}^d$.
\end{proof}

	\noindent{\bf Claim 2.} For any non-empty open subsets $U^{(i)}_0,\dots,U_d^{(i)}$ of $X_{i,\infty}$ with $\bigcap_{j=0}^d\pi_i(U_j^{(i)})\ne\emptyset$, there exists $n\in\N$ such that
	$$U^{(i)}_0 \cap T^{-p(n)} U^{(i)}_1 \cap \cdots \cap T^{-p_d(n)} U^{(i)}_d \neq \emptyset,  i=1,2.$$

\begin{proof}
	Since $(X_{i,l_i},T_i)$ be the $d$-step wTCF for $C$ of $(X_{i,\infty},T_i)$ where $X_{i,\infty}$ is a $\infty$-step pro-nilsystem (so $Z_l(X_{i,\infty})=X_{i,l}$), it follows from Theorem \ref{lei05C} that $$\|\lim_{N\ra\infty}\dfrac{1}{N}\sum_{n=0}^{N-1}(\prod_{j=1}^dT_i^{p_j(n)}f_j^{(i)}-\prod_{j=1}^dT_i^{p_j(n)}\E(f_j^{(i)}|\mathcal{Z}_{i,l_i}))\|_{L^2(X_{i,\infty})}=0$$
	where $f^{(i)}_j\in L^\infty(X_{i,\infty})$. 
	Since for any $F\in L^\infty(X_{1,l_1}\times X_{2,l_2})$, there exist $h^{i}_n\in L^\infty(X_i)$ such that $\|h^1_n h^2_n\ra F\|_{L^2(X_1\times X_2)}\ra 0$. Then we have
	$$\|\lim_{N\ra\infty}\dfrac{1}{N}\sum_{n=0}^{N-1}(\prod_{j=0}^dT^{p_j(n)}F_j^{(i)}-\prod_{j=0}^dT^{p_j(n)}\E(F_j^{(i)}|\mathcal{Z}_{1,l_1}\times \mathcal{Z}_{2,l_2}))\|_{L^2(X_{1,\infty}\times X_{2,\infty})}=0$$
	for any $F^{(i)}_j\in L^\infty(X_{1,\infty}\times X_{2,\infty})$ where $T=T_1\times T_2$. 
	
	Let $(x_j^{(1)},x_j^{(2)})_{j=0}^d\in R_{\pi_1\times\pi_2}^{d+1}=\text{Supp}(\lambda_{d+1})$ and let $	U:=(U^{(1)}_j,U_j^{(2)})_{j=0}^d$ be a neighborhood of $(x_j^{(1)},x_j^{(2)})_{j=0}^d$, then it follows that
	$$\lambda_d(U)= \int_{Z_{1,l_1}\times Z_{2,l_2}} \prod_{j=1}^d \E(1_{U^{(1)}_j}\times 1_{U^{(2)}_j}|\mathcal{Z}_{1,l_1}\times \mathcal{Z}_{2,l_2}) \, d(\mu_{1,l_1}\times\mu_{2,l_2} ) >0.$$
	Thus, $$\begin{aligned}
		& \lim_{N\to\infty} \frac{1}{N} \sum_{n=0}^{N-1} \mu_1\times\mu_2\bigl((U^{(1)}_0,U^{(2)}_0)\cap \bigcap_{j=1}^{d} T^{-p_j(n)}(U^{(1)}_j,U^{(2)}_j)\bigr) \\
		=& \lim_{N\to\infty} \frac{1}{N} \sum_{n=0}^{N-1} \int_{X_1\times X_2} 1_{U^{(1)}_0\times U^{(2)}_0}(x) \cdot\prod_{j=1}^d 1_{U^{(1)}_j\times U^{(2)}_j}(T^{p_j(n)}(x_1,x_2))\, d(\mu_1\times\mu_2)(x_1,x_2) \\
		=& \lim_{N\to\infty} \frac{1}{N} \sum_{n=0}^{N-1} \int_{Z_{1,l_2}\times Z_{2,l_2}} 
		1_{U^{(1)}_0\times U^{(2)}_0}(x) \cdot\prod_{j=1}^d \E(1_{U^{(1)}_j\times U^{(2)}_j}|\mathcal{Z}_{1,l_1}\times \mathcal{Z}_{2,l_2})(z_1,z_2)\\
		&\ \ \ \ \ \ \ \ \ \ \ \ \ \cdot\prod_{j=1}^d\E(1_{U^{(1)}_j\times U^{(2)}_j}|\mathcal{Z}_{1,l_1}\times \mathcal{Z}_{2,l_2})(T_1^{p_j(n)}z_1,T_2^{p_j(n)}z_2))\, d(\mu_{1,l_1}\times\mu_{2,l_2} )(z_1,z_2)
		\\
		\ge& \liminf_{N\to\infty} \frac{1}{N} \sum_{n=0}^{N-1} a^{d+1} \int_{Z_{1,l_2}\times Z_{2,l_2}} 1_{A_a}(z_1,z_2)\prod_{j=1}^d 1_{A_a}(T_1^{p_j(n)}z,T_2^{p_j(n)}z)\, d(\mu_{1,l_1}\times\mu_{2,l_2} )(z_1,z_2) \\
		=& \liminf_{N\to\infty} \frac{1}{N} \sum_{n=0}^{N-1} a^{d+1} \mu_{1,l_1}\times\mu_{2,l_2}\bigl(A_a \cap T^{-p_1(n)}A_a \cap T^{-p_2(n)}A_a \cap \cdots \cap T^{-p_d(n)}A_a\bigr),
	\end{aligned}$$
	where $a>0$ and
	$$
	A_a = \bigl\{ (z_1,z_2) \in Z_{1,l_2}\times Z_{2,l_2} : \E(1_{U^{(1)}_j\times U^{(2)}_j}|\mathcal{Z}_{1,l_1}\times \mathcal{Z}_{2,l_2})(z_1,z_2)> a ,\forall 1\leq j\leq d\bigr\}.
	$$
	As $\E(1_{U^{(1)}_j\times U^{(2)}_j}|\mathcal{Z}_{1,l_1}\times \mathcal{Z}_{2,l_2}) \le 1$, $j=0,1,\dots,d$, by Theorem \ref{lei05C} we can get that
	\begin{align*}
		0 < b :=& \int_{Z_{1,l_2}\times Z_{2,l_2}}\prod_{j=0}^d \E(1_{U^{(1)}_j\times U^{(2)}_j}|\mathcal{Z}_{1,l_1}\times \mathcal{Z}_{2,l_2})\, d(\mu_{1,l_1}\times\mu_{2,l_2} )\\
		=&\left( \int_{A_a}+\int_{Z_{1,l_2}\times Z_{2,l_2} \setminus A_a}\right) \E(1_{U^{(1)}_j\times U^{(2)}_j}|\mathcal{Z}_{1,l_1}\times \mathcal{Z}_{2,l_2})\, d(\mu_{1,l_1}\times\mu_{2,l_2} )\\
		\le& \mu_{1,l_1}\times\mu_{2,l_2}(A_a) + a\mu_{1,l_1}\times\mu_{2,l_2} (Z_{1,l_2}\times Z_{2,l_2} \setminus A_a) \\
		=& a + (1-a)\mu_{1,l_1}\times\mu_{2,l_2} (A_a).
	\end{align*}
	Now we take $a>0$ such that $\mu_{1,l_1}\times\mu_{2,l_2}(A_a)>0$, it follows from Theorem \ref{rec th} that $$\lim_{N\to\infty} \frac{1}{N} \sum_{n=0}^{N-1} \mu_1\times\mu_2\bigl((U^{(1)}_0,U^{(2)}_0)\cap \bigcap_{j=1}^{d} T^{-p_j(n)}(U^{(1)}_j,U^{(2)}_j)\bigr)>0.$$
	Then there exists $n\in\N$ such that
	$$U^{(i)}_0 \cap T^{-p(n)} U^{(i)}_1 \cap \cdots \cap T^{-p_d(n)} U^{(i)}_d \neq \emptyset, i=1,2.$$
\end{proof}

Then it follows from Proposition \ref{wTCF} that  $(X_{1,l_1}\times X_{2,l_2},T_{1,l_1}\times T_{2,l_2})$ is a $d$-step wTCF for $C$ of $X_{1,\infty}\times X_{2,\infty}$. This completes the proof.    
\end{proof}

\begin{proof}[Proof of Theorem \ref{generalization of main1}]

Let $\pi^{(i)}_{k,j}$ be the factor map from $X_k$ to $X_{j}$ with $k,j\in\N\cup\{\infty\}, k> j$. Let $V_0^{(i)},\dots,V_d^{(i)}$ be open sets of $X_i$ such that  $\bigcap_{j=0}^{d}\pi^i_{\infty}(V_j^{(i)})\ne\emptyset$. Suppose that there are open sets  $U_j^{(i)}\in X_{i,\infty}$ such that $U^{(i)}_j\subset \pi^{(i)}_{\infty}(V_j^{(i)})$ and we have $\bigcap_{j=0}^{d}\pi^{(i)}_{i,l_i}(U_j^{(i)})\ne\emptyset,\forall i\in[k]$. 
Since $(\prod_{i=1}^{k}X_{i,l_i},\prod_{i=1}^{k}T_i)$ is a $d$-step wTCF for $C$ of the product system $(\prod_{i=1}^{k}X_{i,\infty},\prod_{i=1}^{k}T_i)$, by Proposition \ref{wTCF} there is some $n_1 \in \mathbb{N}$ such that
\[
U^{(i)}_0 \cap T^{-p(n_1)} U^{(i)}_1 \cap \cdots \cap T^{-p_d(n_1)} U^{(i)}_d \neq \emptyset,\forall i\in[k].
\]
Then we have
\[
\begin{aligned}
 \pi_\infty^{(i)}(V^{(i)}_0) &\cap \pi_\infty^{(i)}(T^{-p_1(n_1)} V^{(i)}_1) \cap \cdots \cap \pi_\infty^{(i)}(T^{-p_d(n_1)} V^{(i)}_d) \\
& \supseteq U^{(i)}_0 \cap T^{-p_1(n_1)} U^{(i)}_1 \cap \cdots \cap T^{-p_d(n_1)} U^{(i)}_d \neq \emptyset,\forall i\in[k].
\end{aligned}
\]
As $R_{\pi_{\infty,l_i}}\subset \RP^{[\infty]}_{X_\infty}$, it follows Lemma \ref{keyopen} that there is some $n_2 \in \mathbb{N}$ such that
$$V^{(i)}_0 \cap T^{-p'_1(n_2)} (T^{-p_1(n_1)}V^{(i)}_1) \cap \cdots \cap T^{-p'_d(n_2)} (T^{-p_d(n_1)}V^{(i)}_d) \neq \emptyset,\forall i\in[k]$$
where $p'_j(n):=p_j(n+n_1)-p_j(n_1)$. i.e.,
$$V^{(i)}_0 \cap T^{-p_1(n_1+n_2)} V^{(i)}_1 \cap \cdots \cap T^{-p_d(n_1+n_2)} V^{(i)}_d \neq \emptyset,\forall i\in[k].$$
By Proposition \ref{wTCF} we can deduce that $(\prod_{i=1}^{k}X_{i,l},\prod_{i=1}^{k}T_i)$ is a $d$-step wTCF for $C$ of $X$.

\end{proof}

\section{Application of the theorem}
	 
	Let $\mathcal{A}:=\{p_1,p_2,\cdots,p_n,\dots\}$ be a family of distinct polynomials vanishing at $0$, for $d\in\N,n\in\Z$ we define the set 
	$\{p_{1}(n), p_{2}(n), \ldots, p_{d}(n)\}$ as a \emph{$\A_d$-set}. 
	Let $\mathcal{F}_{d\text{-}\A}$ denote the family of sets that contain infinitely many $\A_d$-sets. 
	Furthermore, let $\mathcal{F}_{\infty\text{-}\A}$ 
	denote the family of sets that contain at least one $\A_d$-set for each $d\in\mathbb{N}$. For $k\geq 2$ and a tuple $\mathbf{ x}=(x_1,\ldots , x_k) \in X^k$, $d\in \mathbb{N}\cup \{\infty\}$,
	we shall denote  the set of all independent $k$-tuples along $\mathcal{F}_{d\text{-} \A}$ 
	by $\mathrm{Ind}_{d \text{-} \A}^{(k)}(X)$.

    \medskip
	As an appplication of Theorem \ref{main thm} we are able to show the following theorem.
	\begin{thm}\label{App}
		Let $\mathcal{A}:=\{p_1,p_2,\cdots,p_n,\dots\}$ be a family of distinct polynomials vanishing at $0$ and let $(X,T)$ be a minimal system.
		Then there exists a dense $G_\delta$ subset $\Omega$ of $X$ such that for any $x\in \Omega$ and any $k\ge 2$,
		\begin{equation}
			(\mathbf{RP}^{[\infty]}[x])^{(k)}\subseteq Ind_{\infty \text{-} \A}^{(k)}(X),
		\end{equation}
		where $\mathbf{RP}^{[\infty]}[x]=\{x'\in X:(x,x')\in \mathbf{RP}^{[\infty]}\}$.
	\end{thm}

	\begin{lem}\label{ap-open-d-step}
			Let $C=\{p_1,\dots,p_d\}\subseteq\A$ and let $\pi:(X,T)\to (Y,T)$ be a factor map of minimal systems and $d,l\in\mathbb{N}$.
			If  $\pi$ is open and $R_{\pi}\subseteq \mathbf{RP}_X^{[\infty]}$, then
			for any open sets $U_0^{(i)},U_1^{(i)},\ldots,U_d^{(i)}\subseteq X$ with
			$\bigcap_{j=0}^d\pi(U_j^{(i)})\neq\emptyset$ for every $i\in [l]$,
			there exists $n\in \mathbb{N}$ such that
			\[
			U_0^{(i)}\cap T^{-p_1(n)}U_1^{(i)}\cap \cdots\cap T^{-p_d(n)}U_d^{(i)}\neq\emptyset,\quad \forall\; i\in[l].
			\]
			i.e.,\[\prod_{i=1}^{l}U_0^{(i)}\cap \widetilde{T}^{-p_1(n)}\prod_{i=1}^{l}U_1^{(i)}\cap \cdots\cap \widetilde{T}^{-p_d(n)}\prod_{i=1}^{l}U_d^{(i)}\neq\emptyset,\] where $\widetilde{T}=T\times \cdots\times T$ ($l$-times).
		\end{lem}
		\begin{proof}
			By taking $\pi_i=\pi$, it is follows from Lemma \ref{keyopen} directly.
		\end{proof}
		
		\begin{lem}\label{open-infty-step}
			Let $\pi:(X,T)\to (Y,T)$ be a factor map of minimal systems and $d\in\mathbb{N}\cup\{\infty\}$.
			If  $\pi$ is open and $R_{\pi}\subseteq  \mathbf{RP}_X^{[\infty]}$, then $R_{\pi}^{(k)}\subseteq Ind_{\infty\text{-}\A}^{(k)}(X)$
			for any $k\geq 2$,
			where 	$R_\pi^{(k)}=\{(x_1,\ldots,x_k)\in X^k:\pi(x_1)=\cdots=\pi(x_k)\}$.
		\end{lem}
		\begin{proof}
			it suffices to show that
			for any $k\geq 2$, any open subsets $U_1,\ldots,U_k$ of $X$ with
			$\bigcap_{i=1}^k\pi(U_i)\neq\emptyset$,
			one has that $\mathrm{Ind}(U_1,\ldots,U_k)\in \mathcal{F}_{\infty\text{-}\A}$.

			Now fix an integer $k\geq 2$
			and let $U_1,\ldots, U_{k}$ be open subsets of $X$ with $W=\bigcap_{i=1}^k\pi(U_i)\neq\emptyset$.
			By replacing each $U_i$ with $U_{i}\cap \pi^{-1}(W)$,
			we may assume that $\pi(U_1)=\cdots=\pi(U_{k})=W$. 
			We will inductively  construct the sequence 
			$\{n_i\}_{i\in\mathbb{N}}\subseteq\mathbb{N}$
			such that
			\[
			\bigcup_{i=1}^{\infty}\{ p_1(n_i), p_2(n_i), \ldots, p_i(n_i)\}\in {\rm Ind}(U_1,\ldots, U_k).
			\]

		\noindent{\bf Step 1}:
		Notice that $R_{\pi}\subseteq \mathbf{RP}_X^{[\infty]}$,
		we can apply Lemma \ref{ap-open-d-step} with $ l= k^2$ and $C_1=\{p_1\}$ with $p_1$ vanishing at 0 to the following open sets:
		\begin{align*}
			& U_1\times \cdots \times U_1\times U_2\times \cdots \times U_2\times \cdots\times  U_k\times \cdots \times U_k,\\
			& U_1\times \cdots \times U_k\times U_1\times \cdots \times U_k\times \cdots\times  U_1\times \cdots \times U_k,
		\end{align*}
		to derive  $n_1\in\mathbb{N}$ such that
		\begin{equation}\label{neq-empty}
			U_{i}\cap T^{-p_1(n_1)}U_{j}\neq\emptyset, \quad ~~\forall ~ i, j\in\{1,\ldots,k\}.
		\end{equation}

		\medskip
		\noindent{\bf Step 2}:
		For $(i,j,s,t)\in\{1,\ldots,k\}^{4}$,
		let
		\[
		V_{0}^{(i,j,s,t)}=U_{i}\cap T^{-p_1(n_1)} U_{j},\quad V_{1}^{(i,j,s,t)}=U_{s},\quad V_{2}^{(i,j,s,t)}=U_{t},
		\] which satisfies
		\[ \pi (V_{0}^{(i,j,s,t)})\cap \pi (V_{1}^{(i,j,s,t)})\cap \pi (V_{2}^{(i,j,s,t)})=\pi(U_{i}\cap T^{-n_1} U_{j})\neq\emptyset \]
		by \eqref{neq-empty}.
		Notice that $R_{\pi}\subseteq \mathbf{RP}_X^{[\infty]}$,
		we can apply Lemma  \ref{ap-open-d-step} with $C_2=\{p_1,p_2\}\subseteq \A, l=k^4$ to the following open sets:
		\[
		\prod_{(i,j,s,t)\in\{1,\ldots,k\}^{4} }V_{0}^{(i,j,s,t)}, 
		\prod_{(i,j,s,t)\in\{1,\ldots,k\}^{4} }V_{1}^{(i,j,s,t)}, 
		\prod_{(i,j,s,t)\in\{1,\ldots,k\}^{4} }V_{2}^{(i,j,s,t)},
		\]
		to get some $n_2\in\mathbb{N}$ such that
		\[ V_{0}^{(i,j,s,t)}\cap T^{-p_1(n_2)} V_{1}^{(i,j,s,t)}\cap T^{-p_2(n_2)}V_{2}^{(i,j,s,t)}\neq\emptyset, \ \ \forall\; (i,j,s,t)\in\{1,\ldots,k\}^{4}.\]
		That is for any $i,j,s,t\in [k]$,
		\[ U_{i}\cap T^{-p_1(n_1)} U_{j}\cap T^{-p_1(n_2)}U_{s}\cap T^{-p_2(n_2)}U_{t}\neq\emptyset.\]

		\medskip
		\noindent{\bf Step $r$}:
		Let $r\geq 2$ be an integer and
		assume that we have already found $n_1,\ldots, n_{r-1}\in\mathbb{N}$ such that
		\[
		\bigcup_{i=1}^{r-1}\{ p_1(n_i), p_2(n_i), \ldots, p_i(n_i)\}\in {\rm Ind}(U_1,\ldots, U_k).
		\]
		That is, for any $\eta\in\{1,\ldots,k\}^{\{0,1,\ldots, \frac{(r-1)r}{2} \}}$,
		\[
		U_{\eta}=
		U_{\eta(0)}\cap\bigcap_{i=1}^{r-1}\big(T^{-p_1(n_i)}U_{\eta(\frac{(i-1)i}{2}+1)}\cap\cdots\cap T^{-p_i(n_i)}U_{\eta(\frac{(i-1)i}{2}+i)} \big)
		\neq\emptyset.
		\]
		\medskip
		Now for $\eta\in\{1,\ldots,k\}^{\{0,1,\ldots, \frac{(r-1)r}{2}\}}$ and $\xi\in\{1,\ldots, k\}^{\{1,\ldots,r\}}$, let
		\[
		V_{0}^{(\eta,\xi)}=U_{\eta},\;
		V_{1}^{(\eta,\xi)}=U_{\xi(1)}, \ldots, V_{r}^{(\eta,\xi)}=U_{\xi(r)}.
		\]
		Then
		\[
		\pi(V_{0}^{(\eta,\xi)})\cap \pi( V_{1}^{(\eta,\xi)})\cap \cdots\cap \pi(V_{r}^{(\eta,\xi)})=
		\pi(V_{0}^{(\eta,\xi)})=\pi(U_{\eta})\neq\emptyset.
		\]
		
		Notice that $R_{\pi}\subseteq \mathbf{RP}_X^{[\infty]}$,
		we can apply Lemma \ref{ap-open-d-step} with $C_r=\{p_1,p_2,\cdots,p_r\}\subseteq\A$ and $l= \frac{r(r+1)}{2}+1$ to the following open sets:
		\[
		V_0^{(\eta,\xi)}, V_{1}^{(\eta,\xi)}  ,\ldots, V_{r}^{(\eta,\xi)},
		\]
		to find some ${n_{r}}\in\mathbb{N}$ such that
		\[ V_{0}^{(\eta,\xi)}\cap T^{-p_1(n_{r})} V_{1}^{(\eta,\xi)}\cap T^{-p_2(n_{r})}V_{2}^{(\eta,\xi)}\cap\cdots\cap
		T^{-p_r(n_r)}V_{r}^{(\eta,\xi)}\neq\emptyset\]
		for any $\eta\in\{1,\ldots,k\}^{\{0,1,\ldots, \frac{(r-1)r}{2} \}}$ and
		$\xi\in\{1,\ldots, k\}^{\{1,\ldots,r\}}$.

		This shows that
		\[ U_{\eta}\cap T^{-p_1(n_{r})} U_{\xi(1)}\cap T^{-p_2(n_{r})} U_{\xi(2)}\cap\cdots\cap T^{-p_r(n_{r})}U_{\xi(r)}\neq\emptyset.\]
		Thus $\bigcup_{i=1}^{r}\{p_1(n_i), p_2(n_i),\ldots,p_i(n_i)\}\in {\rm Ind}(U_1,\ldots, U_{k})$.
		
		We finish the construction by induction and hence end the proof.
		
	\end{proof}
	
	At last, we  show Theorem \ref{App},
	\begin{proof}
		By Theorem \ref{hsy} there exist minimal systems $X^*$ and $X_{\infty}^*$ which are almost one to one
		extensions of $X$ and $X_{\infty}$ respectively, and a commuting diagram below such that $\pi^*$ is open.
		\[
		\xymatrix{
			X \ar[d]_{\pi}  & X^* \ar[d]^{\pi*}  \ar[l]_{\sigma} \\
			X_{\infty}   & X_{\infty}^*      \ar[l]_{\tau}
		}
		\]

		Let $\Omega_{X_\infty}=\{y\in X_\infty: |\tau^{-1}(y)|=1\}$ and $\Omega_{X}=\pi^{-1}(\Omega_{X_\infty})$.
		Then $\Omega_{X_\infty},\Omega_X$ are dense $G_\delta$ subsets of $X_\infty$ and $X$ respectively. We next show that $\Omega_X$ meets our requirement.

		Fix  $k\geq 2$ and let $x\in \Omega_X$. Assume that $x_1,\ldots, x_{k}\in \mathbf{RP}_X^{[\infty]}[x]=R_{\pi}[x]$.
		For $i\in [k]$, choose $x_i^*\in X^*$ such that $\sigma(x_i^*)=x_i$.
		Notice that
		\[
		\tau(\theta(x_i^*))=\pi(\sigma(x_i^*))=\pi(x_i)=\pi(x) \in \Omega_{X_\infty},
		\]
		we have
		$\pi^*(x_1^*),\ldots,\pi^*(x_k^*)\in \tau^{-1}(\pi(x))$
		and thus $\pi^*(x_1^*) = \cdots = \pi^*(x_k^*)$.

		As $\sigma$ is almost one to one, $X_\infty$ is also the maximal $\infty$-step pro-nilfactor of $X^*$.
		It follows from Lemma \ref{open-infty-step} that
		$R_{\pi^*}^{(k)}\subseteq \mathrm{Ind}_{d\text{-} \A}^{(k)}(X^*)$,
		which implies that
		$(x_1^*,\ldots,x_k^*)\in  \mathrm{Ind}_{d\text{-}\A}^{(k)}(X^*)$.
		Therefore, by Proposition \ref{factor-ind} we have $(x_1, \ldots, x_k) \in \mathrm{Ind}_{d\text{-}\A}^{(k)}(X)$. 
		This completes the proof.
	\end{proof}

\end{document}